\documentclass[12pt]{amsart}

\def\eps{{\varepsilon}}

\def\Leb{{\rm Leb}}

\def\Tor{\mathbb{T}}

\def\reals{\mathbb{R}}

\def\integers{\mathbb{Z}}

\def\bc{\mathbf{c}}

\def\bx{\mathbf{x}}

\def\btau{{\boldsymbol{\tau}}}

\def\brC{{\bar C}}

\def\brH{{\bar H}}

\def\cC{\mathcal{C}}

\def\cD{\mathcal{D}}

\def\cF{\mathcal{F}}

\def\cG{\mathcal{G}}

\def\cN{\mathcal{N}}

\def\cP{\mathcal{P}}

\def\cT{\mathcal{T}}

\makeatother

\usepackage{xcolor}

\usepackage{graphicx,psfrag,epsfig,multirow}

\usepackage[margin=3cm]{geometry}
\def\eps{\varepsilon}
\usepackage[normalem]{ulem}
\usepackage{amsthm, bbm}
\usepackage{amsmath, amssymb}
\usepackage{mathrsfs}
\usepackage{hyperref}
\usepackage{graphicx, psfrag, epstopdf}

\newtheorem{Main}{Theorem}

\theoremstyle{definition}
\newtheorem{definition}{Definition}

\newtheorem{lemma}[definition]{Lemma}
\newtheorem{corollary}[definition]{Corollary}
\newtheorem{proposition}[definition]{Proposition}
\newtheorem{problem}[definition]{Problem}

\numberwithin{equation}{section}

\definecolor{OliveGreen}{rgb}{0,0.6,0}

\definecolor{bpurple}{rgb}{0.74,0.2,0.64}

\def\DS{\displaystyle}
\def\R{\mathbb{R}}
\def\Z{\mathbb{Z}}
\def\T{\mathbb{T}}
\def\N{\mathbb{N}}
\def\a{{\bf a}}

\def\dima2color{\color{purple}}

\title[Smooth zero entropy flows satisfying the classical CLT]
{Smooth zero entropy flows satisfying the classical central limit theorem}
\author{Dmitry Dolgopyat, Bassam Fayad, Adam Kanigowski}
\begin{document}
\maketitle
\begin{abstract}
We construct conservative analytic flows of zero
metric entropy which satisfy the  classical central limit theorem.  
\end{abstract}

\section{Introduction}

Let $(M,\zeta)$ denote a smooth orientable manifold $M$ with a smooth  measure $\zeta$ and
 $F_T$ be a $C^r$ flow on $M$ preserving $\zeta$. When $\zeta$ is a volume measure, we say that the flow $(F_T)$ is conservative.

In this paper we  work exclusively with flows and so the definitions below will be stated for flows. The definitions are analogous for diffeomorphisms with obvious modifications. Following \cite{DDKN}, we define the class of flows satisfying the Central Limit Theorem as follows:

\begin{definition}
\label{DefCLT}  Let $r\in (0,\infty]$. We say that a  flow $(F_T)\in C^r(M,\zeta)$ satisfies the
{\em Central Limit Theorem (CLT)} on $C^r$ if 
 there is a function $a:\R_+\to \R_+$ such that
 for each $A \in  C^r(M)$, 
 $$\frac{\DS \int_{0}^{T}A\circ F_s(\cdot)ds-T\cdot \zeta(A)}{a_T}$$ converges in law as $T\to\infty$
to normal random variable with zero mean and variance $\sigma^2(A)$ 
(such normal random variable will be denoted $\cN(0, \sigma^2(A))$) and, moreover, 
$\sigma^2(\cdot)$
is not identically equal to zero on $C^r(M).$
We say that $F$ satisfies the {\em classical CLT}
if one can take $a_T=\sqrt{T}.$ 
\end{definition}

In this definition we used for an integrable function $A$ on $M$,  the notation $\zeta(A):=\int_M A(x)d\zeta(x)$.

In \cite{DDKN} the authors constructed for every $r\in \N$, examples of conservative $C^r$ diffeomorphisms and flows of zero entropy  satisfying the classical CLT.  However the dimension of the manifold supporting such flows is a linear function of $r$ and so it goes to $\infty$ as $r\to \infty$. In particular the class of zero entropy systems proposed in \cite{DDKN} do not yield $C^\infty$ examples (see the end of the introduction below for more on this). In the current paper we  address, in the context of flows, the $C^\infty$ and real analytic cases. 

\begin{Main}\label{thm:main2} 
There exists a real analytic compact manifold $M$ with a real analytic volume measure $\zeta$ 
and a flow $(F_T) \in C^\omega(M,\zeta)$ that has zero metric entropy and satisfies  the classical CLT. 
\end{Main}

We point out that in the examples we will construct to prove the above theorem, the flows will be analytic
 but the CLT will hold for all sufficiently smooth observables (of class $C^3$). 

 It is still an open problem to find $C^\infty$, zero entropy diffeomorphism which satisfies the classical CLT.  In Section \ref{ScCLT-Par}
 we will explain the reason why our construction does not extend simply to $\Z$ actions.

Similarly to \cite{DDKN}, the example we will construct to prove Theorem \ref{thm:main2} belong to the class of generalized $(T,T^{-1})$-transformations which we now define. We will do so in terms of flows, the definitions for diffeomorphisms being analogous. 

\begin{definition}\label{def:TTinv} 
Let $(K_T)_{T\in \R}$ be a $C^r$-flow, $r\in \N^*\cup \{\infty, \omega\}$, on a manifold $X$ preserving a smooth measure $\mu$ and let $\tau=(\tau_1,\ldots,\tau_d):X\to \R^d$ be a 
 $C^r$ function (called a {\em cocycle} in what follows). Let $(G_t)_{t\in \R^d}$ be an  $\R^d$ action of class $C^r$ on a manifold $Y$ preserving a smooth measure $\nu$. Set
\begin{equation}\label{eq:Tcont}
F_T(x,y)=(K_T(x),G_{\tau_T(x)}y)
\quad\text{where}\quad \tau_T(x)=\int_0^T\tau(K_s x)ds.
\end{equation}
 Then $(F_T)$ is a $C^r$ flow on $M=X\times Y$ preserving the smooth measure $\zeta=\mu\times \nu$.
\end{definition}

Note that by \cite[Lemma 2.1]{DDKN}
if 
the metric entropy of $(K_T,\mu)$ vanishes and $\mu(\tau_i)=0$ for every $i\in [1,d]$, then 
the metric entropy of $(F_T,\zeta)$ is zero\footnote{ \cite[Lemma 2.1]{DDKN} follows
from Ruelle inequality and the fact that the Lyapunov exponents of $F_T$ are zero.}.

On the other hand, the {\em topological} entropy of $F_T$ in our example is positive. In contrast 
in \cite{DDKN} an example is given of a finitely smooth $(T, T^{-1})$ diffeo which satisfies the classical
CLT and has zero topological entropy. In fact, the example in \cite{DDKN} has a rotation in the base and
so the base is uniquely ergodic. In our construction the base map has $N+1$ ergodic invariant measures:
the Lebesgue measure and measures supported at the fixed points. The measures which project to the Dirac
measure on the base but are smooth in the fiber have positive entropy, so the topological entropy of $F_T$
is positive. It is an open problem to construct an analytic flow which satisfies the classical CLT
and has zero topological entropy.

Following \cite{DDKN}, the examples we will give to prove Theorem \ref{thm:main2} are of the form \eqref{eq:Tcont}.  To be more specific, we need to explicit our choices for the flow $(K_T)_{T\in \R}$, 
 the fiber dynamics $(G_t)_{t\in \R^d}$, and the cocycle $\tau$.

On the base we will use area preserving smooth flows on $\T^2$ with degenerate saddles. These belong to the class of conservative surface flows  called {\em Kochergin flows}.
They are the simplest mixing examples of conservative surface flows and were introduced by Kochergin in the 1970s \cite{Koch}. 
 Kochergin flows are time changes of linear flows on the $2$-torus with an irrational slope and with finitely many rest points (see Figure \ref{fig.kochergin} and Section \ref{SSOverview}  for a precise definition of Kochergin flows). 
Equivalently these flows can be viewed as special 
flows over a circular irrational rotation and under a ceiling, or roof,  function with at least one power singularity 
\footnote{The special flows are defined in Section \ref{sec.koc}, see equation \eqref{eq.special} 
and Figure \ref{sym}.}

\begin{figure}[htb]
 \centering
  \resizebox{!}{5cm} {\includegraphics[angle=1]{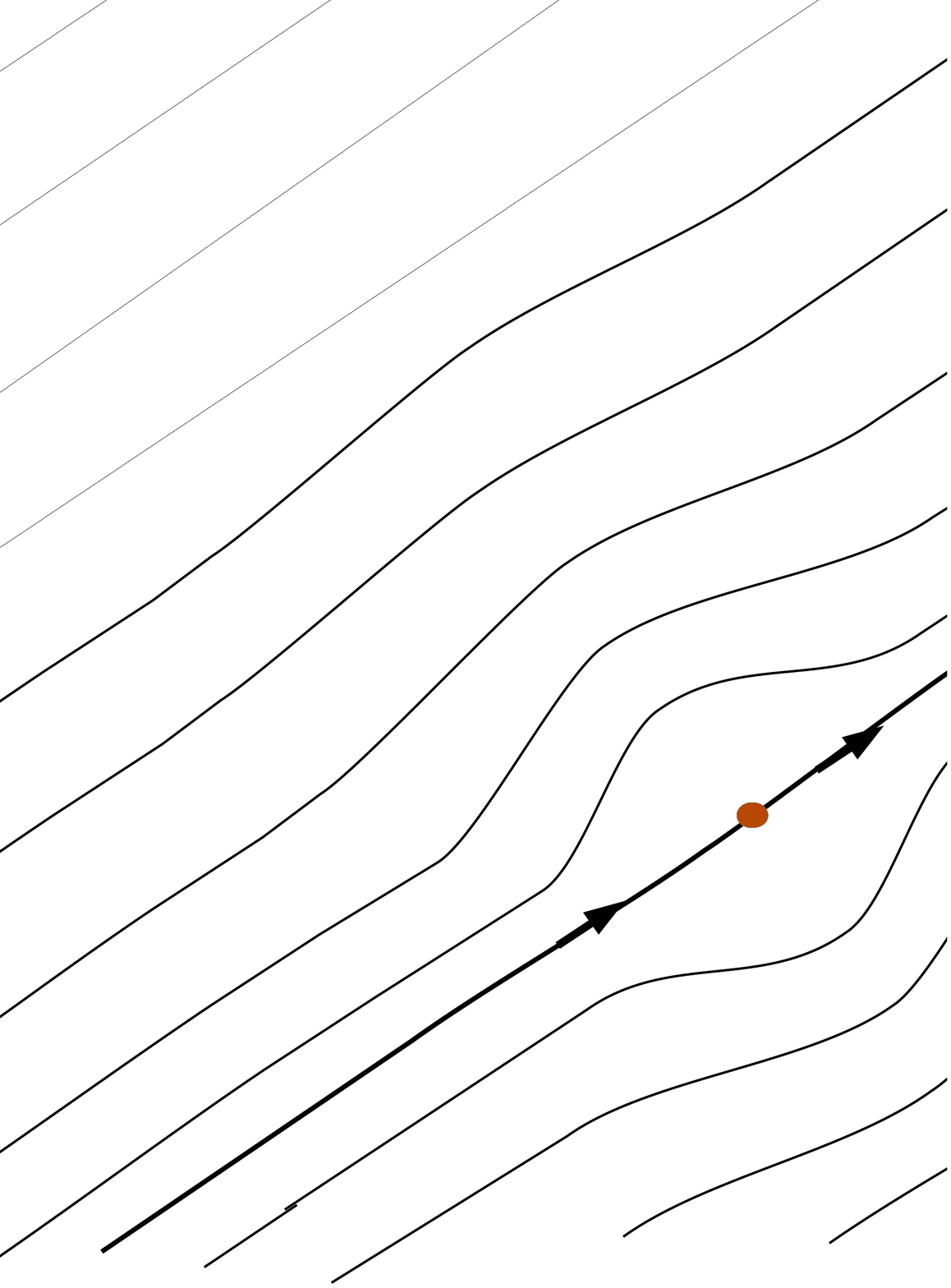}}
\caption{\small Torus flow with one degenerate saddle acting as a stopping point.} 
  \label{fig.kochergin}
\end{figure}

The Kochergin flows that we will consider will have ceiling functions with power singularities of exponent $\gamma\in (0,1/2)$, and will have a rotation number on the base that satisfies a full measure Diophantine type condition. 

For the fiber dynamics, following \cite{DDKN}, we will just need the property of exponential mixing of all orders.  
A classical example of a analytic $\R^d$ action which is exponentially mixing of all orders is the {\em Weyl chamber flow}: Let $d\geq 1$ and let  $\Gamma$ be a co-compact lattice in $SL(d+1,\R)$. Let $D_+$ be the group of diagonal matrices in $SL(d+1,\R)$ with positive elements on the diagonal acting on $SL(d+1,\R)/\Gamma$ by left translation. Then $D_+$ is an $\R^d$ action that preserves Haar measure $\nu$ on $SL(d+1,\R)/\Gamma$  and that is exponentially mixing of all orders. Hence, we can take $G_t$ to be $D_+$.

We are ready now to give a more explicit statement of Theorem \ref{thm:main2} that will be made more precise in Section \ref{sec.koc} after Kochergin flows are precisely defined. We denote $\mu$ the Lebesgue measure on $\T^2$. 

\begin{Main} \label{thm:main22} There exists $N\in \N$ and a Kochergin flow $(K_T,\T^2,\mu)$,  
with $N$ singularities and a function $\tau=(\tau_1,\ldots,\tau_{N}) \in C^\omega(\T^2,\R^{N})$ such that $\mu(\tau_i)=0$, for every $i\in [1,N]$, and such that the flow 
$(F_T)\in C^\omega(\T^2\times (SL(N+1,\R)/\Gamma),\mu\times \nu)$ 
defined by $F_T(x,y)=(K_T(x),G_{\tau_T(x)}y)$ 
satisfies the classical CLT.
\end{Main}

The dimension of the manifold on which our examples are constructed depends thus on the number of singularities $N$ that we require for the Kochergin flow. We did not try to optimize this number, but the one we currently have is of order 100.  


In the next section we will define the class of {\it slowly parabolic} systems and recall the criterion given in \cite{DDKN} that establishes the classical CLT for skew products above a slowly parabolic system, provided the fiber dynamics are exponentially mixing of all orders. This part is essentially the same as in \cite{DDKN}. In a nutshell, {\it slowly parabolic flows} are conservative flows for which the deviations of Birkhoff averages are $o(\sqrt{T})$, but for which there exists $d \in \N^*$, and $d$-dimensional observables whose Birkhoff averages deviate, for every $T$, by more than $(\ln T)^2$ outside exceptional sets of measure less  $o(T^{-5})$.  

The novelty of this note is to prove the existence of smooth (in fact real analytic) conservative flows that are slowly parabolic. We actually show that Kochergin flows on the two-torus, with exponent $\gamma\in (0,1/2)$ for the singularities of their ceiling function, and with $N(\gamma)$ singularities, are  slowly parabolic for typical positions of the singularities and the slope of the flow.

The exponent of a singularity of the ceiling function is related to the order of degeneracy of the corresponding saddle point on $\T^2$. Limiting the order of degeneracy of the saddles thus limits the exponents to be strictly less than $1/2$. This is  important to guarantee that  the deviations of the Birkhoff averages above the Kochergin flow to be $o(\sqrt{T})$. 

The trickiest part of the construction will be to show that if the number of saddles is sufficiently large then we can construct a smooth (and even real analytic) observable $\tau\in C^\omega(\T^2,\R^{N})$  whose Birkhoff averages above the Kochergin flow deviate by more than $(\ln T)^2$ outside exceptional sets of measure less  $o(T^{-5})$.

The Diophantine property imposed on  rotation angle
$\a$ plays a crucial role in insuring refined estimates on Birkhoff sums of functions with singularities above the circular rotation of angle $\a$, which in turn can be used to control the Birkhoff sums of observables above the Kochergin flow. Here again, we did not seek to optimize the Diophantine condition but just made sure it is of full measure.

 It turns out that in finite smoothness $C^r$,  certain ergodic rotations on high dimensional tori (the dimension of the torus goes to $\infty$ with $r$) are examples of diffeomorphisms that satisfy the two conditions on the deviations of the Birkhoff averages, in fact they are {\em slowly parabolic}. For this reason, they could be used in \cite{DDKN} to construct examples of CLT diffeomorphisms with zero entropy in finite smoothness.

\section{ CLT for skew-products above slowly parabolic systems}
\label{ScCLT-Par}

In this section, we describe general conditions on the flow $(K_T,X,\mu)$ which will allow us to construct a generalized $(T,T^{-1})$ flow $(F_T)$ as in Definition \ref{def:TTinv} that satisfies the assumptions of Theorem \ref{thm:main2}. 

\begin{definition}\label{def:parflow} Let $(K_T)_{T\in \R}$ be a $C^r$-flow on a manifold $X$ preserving a smooth measure $\mu$. We say that $(K_T)$ is {\em $C^r$--slowly parabolic}
if the following conditions are satisfied:
\begin{enumerate}
\item[$S1$.] for every $H\in C^r(X)$ with $\nu(H)=0$,\; 
$\DS
\frac{1}{\sqrt{T}}\int_0^T H(K_t\cdot )dt \Rightarrow 0,
$
in distribution as $T\to\infty$.
\item[$S2$.] there exist $C,d\in \N$ and a $C^r$ function  ${\btau}=(\tau_1,\ldots,\tau_d):X\to \R^d$, $\nu(\btau)=0$ such that 
$$
\nu\Big(\{x\in X\;:\; \Big\|\int_0^T\btau(K_tx)dt\Big\|<C\ln^2 T  \}\Big)=o(T^{-5}).
$$
\item[$S3$.] there exist $C>0, m <1.1$  and $x_0\in X$ such that for every $\delta>0$ sufficiently small, we have
$\DS
K_tB(x_0,\delta)\cap B(x_0,\delta)=\emptyset
$
for every $|t|\in (C\delta, (C\delta)^{-1/m})$.
\end{enumerate}
\end{definition}



Conditions $S1$ and $S2$ are used to show that the associated $(T,T^{-1})$-flow satisfies the classical CLT (with the possibility that the variance is identically zero). Condition $S3$ 
insures that there  exists a function with non-zero variance.

The following result based on Theorem 3.2 in  \cite{DDKN} reduces the proof of 
Theorems \ref{thm:main2} and \ref{thm:main22} to that of finding smooth (and real analytic) slowly parabolic flows. 

\begin{proposition}\label{prop:1} Assume that $(K_T,X,\zeta)$ is a $C^\omega$-slowly parabolic flow of zero  metric entropy. Let $(G_t,Y,\nu)$ be an analytic $\R^d$ action which is exponentially mixing of all orders. Let 
$$F_T(x,y)=(K_Tx, G_{\btau_T(x)}(y)),$$ where $\btau$ is as in $S2$. Then $(F_T,X\times Y,\mu \times \nu)$ has zero metric entropy and satisfies the classical CLT. Moreover there exists $H\in C^\omega(X\times Y)$ with $\sigma^2(H)\neq 0$.
\end{proposition}
\begin{proof}
Everything but the last assertion on the non-vanishing of the variance is a direct consequence of Theorem 3.2 in \cite{DDKN}. The fact that the variance is not identically zero follows from $S3$ similarly to the proof of Lemma 8.2 in \cite{DDKN}. For completeness we include the details in the appendix.
\end{proof}

As explained in the introduction a classical example of an analytic $\R^d$ action which is exponentially mixing of all orders is the {\em Weyl chamber flow}.  It remains to find examples of smooth and real analytic slowly parabolic flows. 
 In light of Proposition \ref{prop:1}, Theorem~\ref{thm:main2} becomes an immediate consequence of the following result:
\begin{Main}\label{thm:SP} There exists an analytic conservative flow $(K_T,X,\mu)$ with zero metric entropy  that is a slowly parabolic flow.
\end{Main}
Theorem \ref{thm:SP} is the main novelty of this work.
 
Existence of $C^\infty$ slowly parabolic diffeomorphisms is an open question. In fact, to the best of our knowledge, the following easier problem is open:
\begin{problem} Construct a $C^\infty$  diffeomorphism $f$ on a smooth compact manifold $X$ preserving a smooth measure $\mu$ such that:
\begin{itemize}
\item[T1.] for every $H\in C^\infty (X)$ with $\mu(H)=0$ we have
$\DS
\frac{1}{\sqrt{N}}\sum_{n\leq N} H(f^n \cdot )dt{ \Rightarrow} 0,
$
in distribution as $T\to\infty$;
\item[T2.] there exists $x\in X$ and $\phi\in C^\infty(X)$ such that 
$\DS
\{\phi_n(x)\}:=\Big\{\sum_{n\leq N} \phi(f^nx)\Big\} $
is unbounded. 
\end{itemize}
\end{problem}

In other words, in all the known smooth examples, whenever there exists a zero average function which is not a coboundary (equivalently T2 holds) then there is a rapid jump in asymptotics of ergodic averages, i.e. they become of order $\sqrt{N}$ or larger. Hence, in light of Katok's conjecture on cohomologically rigid diffeomorphisms, one can ask the following:

{\it Does there exist a $C^\infty$ diffeomorphism $f$ on a smooth compact manifold $X$ preserving a smooth measure $\mu$, not conjugated to a Diophantine torus translation}, for which T1 holds?


 It is interesting to point out that the classical parabolic flows,including horocycle flows and their reparametrizations, nilflows and their reparametrizations, {\bf are not} slowly parabolic. Indeed, it follows from the work of Flaminio-Forni, \cite{Fla-Fo} and \cite{Fla-Fo2} that the deviations of ergodic averages in these examples are, for observables that are not coboundaries, of order at least $\sqrt{T}$ for a positive measure set of points.
Thus property T1 does not hold for those flows.
 Moreover these flows are known not to have a CLT and it is therefore not possible to use them to construct skew products above them that satisfy the classical CLT.



The rest of the paper is devoted to the proof of Theorem \ref{thm:SP}. Our examples will belong to the class of smooth flows on surfaces with degenerate saddles (so called {\em Kochergin flows}). 

For the class of Kochergin flows that we consider all the singularities will be ``weakly''  degenerate, i.e. the strength of the singularity will be $o(x^{-1/2})$. This assumption will relatively easily give us the condition $S1$ for any number of singularities. Condition $S3$ will also be easy to achieve by assuming that the base rotation (the first return map) is Diophantine. The most interesting and also most difficult part is to show existence of $\btau$ satisfying the assumptions of $S2$.

\section{ Construction of slowly parabolic  Kochergin flows in the smooth case} \label{sec.koc}

\subsection{Overview of the construction} 
\label{SSOverview}

We start by defining $C^\infty$ Kochergin flows on $\T^2$. They were introduced by Kochergin in \cite{Koch}.  As explained in \cite{FFK}, the construction can be made analytic.

 The construction in \cite{Koch} is to take a linear flow on $\T^2$ in direction $(\alpha,1)$ and cut out finitely many disjoint disc from the phase space. Inside each such disc one then glues in a {\em Hamiltonian flow} on $\R^2$ with a degenerated singularity at $\bar{c}\in \T^2$ (corresponding to $(0,0)\in \R^2$). Finally one smoothly glues the trajectories of the linear flow with the trajectories of the Hamiltonian flow.  It follows that each such flow preserves a smooth area measure on $\T^2$. Moreover, as shown by Kochergin, such flows are mixing for all irrational $\alpha\in \T$. For more details on the construction we refer the reader to \cite{Koch}. In our case we will cut out finitely many discs centered at 
$\{\bar{c}_i\}_{i=1}^N$ and glue a Hamiltonian flow with a degenerated singularity at $0$ in the discs centered  at points $\{\bar{c}_i\}_{i=1}^N$. From the construction it follows that the set $\cT=\T\times\{0\}$ is a global transversal for the flow (we can WLOG assume that no discs intersects $\cT$) and moreover the first return map is the rotation by $\alpha\in \T$. The roof function $f:\T\to \R_+$ (first return time) is smooth except at the  projections along the flow lines of the points $\{c_i\}_{i=1}^N$ at which it has a power-like singularity with exponent $\gamma \in (0,1)$. 

In what follows, when we write $\{c_i\}_{i=1}^N\subset \cT$, we allow singularities of the smooth flow $(K_t)$ to be anywhere on the unit flow lines of the linear flow in direction $\alpha$, i.e.  $\bar{c_i}=L^{\alpha}_{t}(c_i,0)$ with $0<t_i<1$ and where $(L^\alpha_t)$ denotes the linear flow on $\T^2$ in direction $(\alpha,1)$. 
 In fact we will construct   {\em good} tuples of points $(c_1,\ldots, c_N)$ and then we lift them along the flow as described above.

In particular, every point in $x\in M$ which is not a fixed point can be written as $x=K_w \theta$, where $\theta\in \cT$ and $0\leq w<f(\theta)$. By the construction it follows that 
\begin{equation}\label{fsum}
 f(\cdot)=\sum_{i=1}^N \bar{f}(\cdot-c_i),
\end{equation}
where $c_i\in \cT$ denote the projections of $\bar{c}_i\in \T^2$ along the flow lines. Moreover, as shown in \cite{Koch}, $\bar{f}:\T\to \R_+$ is $C^3$ on $\Tor\setminus\{0\}$, satisfies
  $\int \bar f d Leb = 1$ and 
\begin{equation}\label{eq:asy2}
\lim_{\theta\to 0^+}\frac{\bar{f}''(\theta)}{\theta^{-2-\gamma}}=A=
\lim_{\theta\to 1^-}\frac{\bar{f}''(\theta)}{(1-\theta)^{-2-\gamma}}
\end{equation}
where $A>0$ and $\gamma\in (0,1)$.  In this context Kochergin showed that $\gamma=1/3$ is a possible exponent. For simplicity we will always assume that  \eqref{eq:asy2} holds with
$A=1$. Let us denote $\mathcal{K}(\alpha,\gamma, \{c_i\})$ the set of smooth area preserving flows $(K_t)$ on $\T^2$ for which $R_\alpha$ is the first return map and the corresponding first return time $f$ satisfies \eqref{fsum} where $\bar{f}$ satisfies \eqref{eq:asy2}. In what follows we will always assume that $\gamma<1/2$.

Thus Kochergin flows are isomorphic to special flows defined as follows.
The orbit of  of a point $(\theta,u)$, $\theta \in \T$, $u\in [0,f(\theta))$ under the flow $(K_T)$ for positive time $t$ is given by 
 \begin{equation}\label{eq.special} K_t(\theta, u)=(\theta+n\alpha, u+t-S_n(f)(\theta))\end{equation}
 where $S_n(f)$ is the ergodic sum of $f$ and $n((\theta,u),t)$ is the unique integer such that $0\leq u+t-S_n(f)(\theta)<f(\theta+n\a)$. The orbits for negative times are defined similarly.

\begin{figure}[htb]
 \centering
  \resizebox{!}{4cm}{\includegraphics[angle=0]{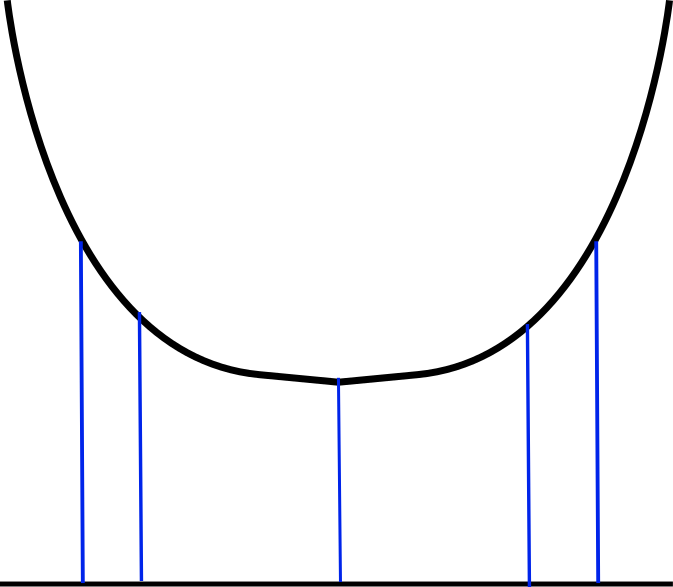}}
\caption{\small  Representation of a $2$-torus flow with one degenerate saddle  as a special flow under a ceiling function with a  power-like singularity.} 
  \label{sym}
\end{figure}

\medskip 

Let $(a_n)$ denote the continued fraction expansion and $(q_n)$ denote the sequence of denominators of $\alpha$, i.e. $q_0=q_1=1$ and 
\begin{equation}\label{eq:qi}
q_{n+1}=a_{n+1}q_n+q_{n-1}.
\end{equation}
Let
\begin{equation}\label{eq:cp}
\cD:=\{\alpha\in \T\;:\;\exists C>0 \; \text{ such that }q_{n+1}<Cq_n \ln^2 q_n \text{ for every }n\in \N\}.
\end{equation}
 The set $\cD$ has full measure by Khintchine's theorem, \cite[Section 13]{Khi}.  The following statement
  gives explicit examples of flows that satisfy Theorem \ref{thm:SP}.
 \begin{Main}\label{prop:sch} For every $\alpha\in \cD$ and every $\gamma<1/2$, there exists $s\in \N$  and a full measure set $\cC\subset [0,1]^s$ such that for every $(c_1,\ldots, c_s)\in \cC$, every $C^r$ flow  in  $\mathcal{K}(\alpha,\gamma, \{c_i\})$ 
 with $r\geq 3$ is $C^r$-slowly parabolic. 
  \end{Main}
In the above theorem $r$ may be equal to $\infty$ or $\omega$. We also point out that the following stronger result holds:  when $r=\omega$, condition $S1$ will still be satisfied for all $C^3$ functions.
Theorem \ref{prop:sch} is  a direct consequence of the following two results.

\begin{proposition}\label{lem:sch2} For every $s\in \N$, every $(c_1,\ldots, c_s)\in [0,1]^s$, every $\alpha\in \cD$ and every $\gamma<1/2$ every smooth flow in  $\mathcal{K}(\alpha,\gamma, \{c_i\})$ satisfies properties $S1$ and $S3$.
\end{proposition}

\begin{proposition}\label{prop:sch2}
For every $\alpha\in \cD$ and every $\gamma<1/2$  there exists $s\in \N$  and a full measure set $\cC\subset [0,1]^s$ such that for every $(c_1,\ldots, c_s)\in \cC$, every smooth flow in  $\mathcal{K}(\alpha,\gamma, \{c_i\})$ satisfies $S2$.
\end{proposition}

We will now prove Proposition \ref{lem:sch2}.
\begin{proof}[Proof of Proposition \ref{lem:sch2}] We first show $S3$. we consider the representation of $K_t$ as a special flow over a rotation (and this identification is smooth away from fixed points). We have 
 $$ K_t(\theta, u)=(\theta+n\alpha, u+t-S_n(f)(\theta)) $$
 for some $|n|\leq \brC |t|$ where $S_n(f)$ is the ergodic sum of $f$. Let $x_0=(\theta,u)$ be any point which is not a fixed point of the flow. 
 If $n=0$,  then $S3$  holds since the second coordinates differ by at least $C\delta$, (as the vector field is positive since $x_0$ is not a fixed point). 
 If $n\neq 0$ then the first coordinates differ by at least $\delta$ since $\alpha\in \cD$. 
 
 Property $S1$ was proven in \cite{DDKN} in case there is just one singularity (i.e. $N=1$). However the proof works with minor changes for multiple singularities. Alternatively, one can also use Theorem 1.1. in \cite{FM} in the context of surfaces of genus $1$. In this case the asymptotic growth in $L^1$ is given in formula  (1.5) of \cite{FM} ({Note that $b(\sigma,|\alpha|)$ appearing in RHS of that formula
 equals to  $\gamma<\frac{1}{2}$ for the case considered in our paper}).
\end{proof}

The proof of Proposition \ref{prop:sch2} is the most important part of the paper. The rest of  Section \ref{sec.koc} as well as Section \ref{ScTechnical} is devoted to its proof. 

In \S \ref{sec.preliminaries}, we give useful estimates on ergodic sums of functions with singularities above circular rotations of frequency $\alpha \in \cD$. 

In \S \ref{sec.singularities}, we give the criteria along which the singularities should be selected.  We state Proposition \ref{prop.choice.c} that says that, under the condition $\alpha \in \cD$,  the typical choice of the singularities satisfies the criteria. 

In \S \ref{sec.tau.choice}, we explain how the cocycle $\tau$ must be chosen. We state the main part of the proof, Proposition \ref{prop.choice.tau} that says that the chosen cocycle satisfies $S2$.

In \S \ref{sec.c.choice}, we prove Proposition \ref{prop.choice.c} and in \S \ref{SSprop.choice.tau} we prove Proposition \ref{prop.choice.tau}. 

Finally, in Section \ref{sec.analytic} we see how the construction can be extended to the real analytic setting.  

In the appendix, following \cite[Lemma 8.2]{DDKN}, we show how $S3$ can be used to construct real analytic observable above $(F_T,M,\zeta)$ with non zero asymptotic variance.

\subsection{Preliminary estimates on ergodic sums of functions with singularities} \label{sec.preliminaries}


\begin{lemma}\label{lem:DK0}Let $\bar{f}$ be as in \eqref{eq:asy2}. Then for every $N\in \N$ and every $x\in \T$
$$
|S_N(\bar{f})(x)- N\int_\T \bar{f}d Leb-\bar{f}(x_{min,N})|={\rm O}\Big(N^\gamma\ln^{5} N\Big),
$$
where $x_{min,N}=\DS \min_{0\leq j<N} \|x+j\alpha\|$.
\end{lemma}
\begin{proof} This follows from the assumptions on $\alpha$ and from Proposition 5.2. in \cite{CW}. 
In fact  \cite[Propositon 5.2]{CW} is proven for monotone observables,
but by our assumptions \eqref{eq:asy2} we can write $\bar{f}=g+h$ where $g$ is monotone and satisfies \eqref{eq:asy2} and $h$ is of bounded variation. 
 The ergodic sums of $g$ are ${\rm O}\Big(N^\gamma\ln^{5} N\Big)$ by \cite{CW}
while the ergodic sums of $h$ are $O(\ln^4 N)$ by Lemma \ref{LmSumBV}
{from Appendix \ref{AppO}}.
\end{proof}

We want to describe the set where ergodic sums of the function $\bar{f}$ satisfying \eqref{eq:asy2} are small. For small enough $\eps>0$ let
\begin{equation}\label{eq:an0}
A_N:=\{x\in \T\;:\; |S_N(\bar{f}_0)(x)|\leq N^{\gamma^2+\eps}\},
\end{equation}
where $\bar{f}_0=\bar{f}-\int_\T \bar{f} d Leb$.

The rest of this section is devoted to the proof of the following :
\begin{proposition} \label{lemma.delta} Let $\delta_N=\frac{1}{N^{1+\gamma/5}}$. There are $a_1,\ldots a_{3N}\subset \T$ such that 
\begin{equation}\label{eq:an}
A_N\subset \bigcup_{i=1}^{3N}(-\delta_N+a_i,a_i+\delta_N)
\end{equation}
holds for every sufficiently large $N\in \N$.
\end{proposition}
\begin{proof}
Let $\bar{f}$ be as in \eqref{eq:asy2}. First we state the following lemma giving lower bounds on the size of $\bar{f}''$:
\begin{lemma}\label{lem:DK} For sufficiently large $N\in \N$ and for every $x\in \T\setminus\{-i\alpha\}_{i=0}^{N-1}$ we have 
$$
S_N(\bar{f}'')(x)\geq N^{2+\gamma}\ln^{-10}N.
$$
\end{lemma}
\begin{proof} From formula (13) in Lemma 4.2 in \cite{FFK}, it follows that $S_N(\bar{f}'')(x)\geq \bar{f}''(x^N_{min})$, where $x^N_{min}$ minimizes the distance from $\{x+j\alpha\}_{j\leq N-1}$ to $0$. It follows that if $n$ is such that  $N\in [q_n,q_{n+1})$, then $0<x^N_{min}<\frac{1}{q_n}$. Then the statement follows from \eqref{eq:asy2} and the diophantine assumption on $\alpha$ as $q_n^{2+\gamma}\geq  q_{n+1}^{2+\gamma}\ln^{-10}(q_{n+1})\geq N^{2+\gamma}\ln^{-10}N$. 
\end{proof}

Consider the interval partition $\cP_N$ of $\T$  by points $\{-i\alpha\}_{i=0}^{N-1}$. Then by \eqref{eq:asy2} it follows that $S_N(\bar{f}_0)(\cdot)$ is a $C^3$ function on Int$(I_j)$ for every interval $I_j\in \cP_N$. Let $I_j=[a_j,b_j)$. By Lemma \ref{lem:DK} it follows that $S_N(\bar{f}')(\cdot)$ is monotone on $I_j$. By d'Hospital's rule 
 $S_N(\bar{f}'_0)(\cdot)$ satisfies $\DS \lim_{x\to a_j^+}S_N(\bar{f}'_0)(x)=-\infty$ and 
 $\DS \lim_{x\to b_j^-}S_N(\bar{f}'_0)(x)=\infty$. So let  $a_{j,1}=y_j\in I_j$ be  the unique 
 point such that $S_N(\bar{f}'_0)(a_{j,1})=0$. Note that 
$$
S_N(\bar{f}'_0)(x)-S_N(\bar{f}'_0)(a_{j,1})=S_N(\bar{f}''_0)(\theta_x)(x-a_{j,1}) 
$$
In particular, if $J_{j,0}$ is an interval of size $\frac{2}{N^{1+\gamma/5}}$ centered at $a_{j,1}$, by Lemma \ref{lem:DK} it follows that for $x\in I_j\setminus J_{j,0}$
\begin{equation}\label{eq:sm1}
|S_N(\bar{f}'_0)(x)|\geq N^{2+\gamma}\ln^{-10}N \cdot \frac{1}{N^{1+\gamma/5}}= N^{1+4\gamma/5}\ln^{-10}N.
 \end{equation}
Consider the two disjoint intervals $K_{j,1}$ and $K_{j,2}$  in $I_j$ so that $K_{j,1}\cup K_{j,2}=I_j\setminus J_{j,0}$. Then for $w=1,2$, let $x_{j,w}\in K_{j,w}$ be the point minimizing  $S_N(\bar{f}_0)(\cdot)$ on $K_{j,w}$. If $S_N(\bar{f}_0)(x_{j,w})\geq N^{\gamma^2+\eps}$, then $A_N\cap I_j\subset J_i$. If not, let $J_{j,w}$ be an interval of size $\frac{2}{N^{1+\gamma/5}}$ centered at $x_{j,w}$. Then for every $x\in K_{j,w}$, we get 
$$
|S_N(\bar{f}_0)(x)-S_N(\bar{f}_0)(x_{j,w})|=|S_N(\bar{f}'_0)(\theta_x||x-x_{j,w}|
$$
for some $\theta_x\in K_{j,w}$. So if $x\in K_{j,w}\setminus J_{j,w}$, then by \eqref{eq:sm1} we get 
$$
|S_N(\bar{f}_0)(x)|\geq N^{1+4\gamma/5}\ln^{-10}N\cdot \frac{1}{N^{1+\gamma/5}}- S_N(\bar{f}_0)(x_{j,w})\geq  N^{3\gamma/5}\ln^{-10}N -N^{\gamma^2+\eps}>N^{\gamma^2+\eps},
$$
if $\eps>0$ is small enough (remember that $\gamma<1/2$). It then follows that 
$$
A_N\cap I_j\subset J_{j,1}\cup J_{j,2}\cup J_{j,0}.
$$
and the $J_{j,w}$ are intervals of size $\frac{2}{N^{1+\gamma/5}}$ centered at $a_{j,w}$. Then 
$\DS
A_N\subset \bigcup_{j=1}^N[J_{j,1}\cup J_{j,2}\cup J_{j,0}].
$
This gives \eqref{eq:an} and finishes the proof.
\end{proof}


\subsection{Choosing the singularities} \label{sec.singularities} 
We will consider a Kochergin flow $(K_t)$ with $\gamma<1/2$, $\alpha\in \cD$, and a roof function $f$ given by 
\begin{equation}\label{fsum'}
 f(\cdot)=\sum_{i=1}^{s+3} \bar{f}(\cdot-c_i),
\end{equation}
where $s:=\frac{50}{\gamma}$. In this subsection we describe the choice of the $c_i\in \cT$. 
In the process, we will explain why it is preferable for the presentation to index the singularities from $1$ to $s+3$ instead of from $1$ to $s$ with a larger $s$. 

For $T>0$ and ${\bf x}=(x,w)\in \T^2$ let $N(x,w,T)$ denote the number of returns of $\{K_t(x,w)\}_{t\leq T}$ to
 $\cT=\T\times \{0\}$,  that is the unique integer $N$ such that 
 $$0\leq w+T-S_Nf(x)<f(x+N\alpha).$$
  
 Let $C=(\inf_\T f)^{-1}$. Notice that  $N(x,w,T)\leq C T$. 
 For $\bar{c}\in \T^2$ Let $V_\kappa (\bar{c})$ be the 
 $\kappa$ neighborhood of $\bar{c}$. 
 
 Recall that $A_N$ is defined in  \eqref{eq:an0}  as 
$
\DS A_N:=\{x\in \T\;:\; |S_N(\bar{f}_0)(x)|\leq N^{\gamma^2+\eps}\}.
$

 For $n\in \N$ and $c\in \T$ let  
\begin{equation}\label{eq:gcn}G(c,n)= \bigcup_{k=-2q_{n+1}}^{2q_{n+1}}
\left[c{+k\alpha}-\frac{1}{q_n\ln^{5}q_n}, c{+k\alpha}+\frac{1}{q_n\ln^{5}q_n}\right].
\end{equation} 

{Let $S(T,(i_v)_{v=1}^s)$ be the set of points $\bx=(x, w)$ such that
\begin{equation}\label{eq:awaysing}
\{{\bf x}, K_T({\bf x})\}\bigcap \bigcup_{j\in \{i_1,\ldots,{i_s}\} }V_\kappa(\bar{c}_j)
{=\emptyset}
\end{equation}
and 
\begin{equation}\label{eq:awaysing2}
 \{x+N(x,w,T)\alpha+u\alpha\}_{|u|\leq 2T^\gamma} \bigcap 
 \bigcup_{j\in \{i_1,\ldots,{ i_s}\}} \Big[-\frac{1}{2T^\gamma\ln^{7}T}+ c_j, c_j+\frac{1}{2T^\gamma\ln^{7}T}\Big]
 {=\emptyset}.
\end{equation}
}

 \begin{definition} \label{def.good} We say that the set of singularities $\{c_1,\ldots,c_{s+3}\}$ is {\it good} if  \begin{itemize} 
\item[$(\cG1)$] For every $i,j \in \{1,\ldots, s+3\}$,  for any $n$  sufficiently large 
$$G(c_i,n)\cap G(c_j,n)=\emptyset$$

\item[$(\cG2)$] For any choice of pairwise different elements  $i_1,\ldots,i_s$  from $\{1,\ldots,s+3\}$ and for $N$ sufficiently large 
 $$Leb_1\Big(\bigcap_{w=1}^sA_N+c_{i_w}\Big)\leq  N^{-6}$$
 
\item[$(\cG3)$] For sufficiently large $T>0$, for every ${\bf x}\in \T^2$ there is $(i_1,\ldots,i_s)\subset \{1,\ldots,s+3\}$ such that $ {\bf x}\in  S(T,(i_v)_{v=1}^s)$ 
\end{itemize} 
  \end{definition} 
 
The set of good tuples will be denoted by $\cC$.
\medskip

The reason why we preferred to index the singularities from $1$ to $s+3$ is that we need, for each 
$\bf x$ and every $T$, that there exist at least $s$ singularities such that $(\cG3)$ holds. In Lemma \ref{lem:sep}, we will see why considering $s+3$ singularities instead of $s$ is sufficient for insuring $(\cG3)$.

The choice of the singularities $c_1,\ldots,c_{s+3}$ is based on the following result that we will prove in Section \ref{sec.c.choice}. 
 \begin{proposition} \label{prop.choice.c}
$Leb_{s+3}(\cC)=1$. 
 \end{proposition}

 \subsection{Choosing the cocycle $\btau$} \label{sec.tau.choice}

Define a $C^\infty$ function $\btau=(\tau_1,\ldots,\tau_{s+3}):\T^2\to \R^{s+3}$ as follows:  For $i\in\{1,\ldots, s+3\}$ let $\tau_i:\T^2\to \R$ be a $C^\infty$ mean zero function such that $\tau_i=1$ on $\kappa$ neighborhood of $\bar{c}_i$ and $\tau_i=0$ on $\kappa$ neighborhoods of $\bar{c}_j$ with $j\neq i$, for some small $\DS 0<\kappa<\frac{1}{2}\min_{i\neq j}\|\bar{c}_i-\bar{c}_j\|$. 
 
 \begin{proposition} \label{prop.choice.tau} If $(c_1,\ldots,c_{s+3})\in [0,1]^{s+3}$ are {\it good}, then the flow $(K_t)$ defined as in \eqref{fsum} satisfies $S2$. More precisely, the function $\btau$ satisfies 
$$
\nu\Big(\{x\in \T^2\;:\; \Big\|\int_0^T\btau(K_tx)dt\Big\|<C\ln^2 T  \}\Big)=o(T^{-5}).
$$
where $\nu$ is the Haar measure on $\T^2$. 
  \end{proposition}

\section{ The proofs of technical propositions.}
\label{ScTechnical}
\subsection{Proof of Proposition \ref{prop.choice.c}} \label{sec.c.choice}
We will always fix $s=\frac{50}{\gamma}$.
We start with $(\cG1)$. 
 \begin{lemma} \label{lemma.G1} There is a set $\cF\in [0,1]^{s+3}$ of full Lebesgue measure such that for any $(c_1,\ldots,c_s)\in \cD$, for every $i,j \in \{1,\ldots, s+3\}$,  for any $n$  sufficiently large 
$$G(c_i,n)\cap G(c_j,n)=\emptyset.$$
\end{lemma} 

\begin{proof} Define
\begin{equation}\label{eq:fn}
F_n:=\Big\{(c_1,\ldots, c_{s+3})\in [0,1]^{s+3}\;:\;  G(c_i,n)\cap G(c_j,n)=\emptyset, \text{ for any } i\neq j\Big\}.
\end{equation}
We have the following: 

\noindent {\sc Claim.} For $n$ sufficiently large, we have
$$
\Leb_{s+3}(F_n)\geq 1-\frac{1}{\ln^2q_n}.
$$

\begin{proof}[Proof of the claim] We say that $d\in \T$ is $n$-far from $c\in \T$ if $G(c,n)\cap G(d,n)=\emptyset$. We have 
$$\
Leb_1(\{d: d \text{ is } n \text{ far from } c\}\geq 1-\frac{8q_{n+1}}{q_n\ln^5q_n}\geq 1-\frac{8{C}}{\ln^3q_n},
$$ the last inequality since $\alpha\in \cD$. Notice that $F_n$ contains vectors $(c_1,\ldots,c_{s+3})\in [0,1]^{s+3}$ such that $c_1$ is any number in $[0,1]$, $c_2$ is $n$-far from $c_1$, $c_3$ is $n$-far from $c_2$ and $n$-far from $c_1$, and so on until $c_s$ is $n$-far from $c_j$ for every $j<s$. Therefore, 
$$
Leb_{s+3}(F_n)\geq \prod_{\ell=1}^{s+3} \left(1-\frac{8{C}\ell}{\ln^3q_n}\right)\geq 1-\ln^{-2}q_n,
$$
if $n$ is sufficiently large.
\end{proof}
Thus the set
 \begin{equation}\label{eq:cf}
\cF:=\bigcup_{m=1}^{+\infty} \bigcap_{n\geq m}F_n
\end{equation}
satisfies $Leb_{s+3}(\cF)=1$ and the condition of Lemma \ref{lemma.G1}.
\end{proof}

We proceed to $(\cG2)$.

 \begin{lemma} \label{lemma.G2}There is a set
  $\mathbb{D}\in [0,1]^s$ of full Lebesgue measure such that for any $(c_1,\ldots,c_s)\in \mathbb{D}$, for $N$ sufficiently large 
 $$Leb_1\Big(\bigcap_{i=1}^sA_N+c_{i}\Big)\leq N^{-6}$$
 \end{lemma}

\begin{proof}
We have  the following, where addition is mod $1$.

\medskip 

\noindent {\sc Claim.} Let $N\in \N$, $\delta>0$, $b_1,\ldots b_N\in \T$ and let $A\subset \T$ be such that\\
 $\DS A\subset \bigcup_{i=1}^N(-\delta+b_i,b_i+\delta)$. Then for any $s\in \N$
\begin{equation}
\label{FewTranslates}
\int_{[0,1]^s}Leb\Big( \bigcap_{i=1}^s(A+t_i)\Big)dt_1\ldots dt_s{\leq} (2N\delta)^{s}.
\end{equation}

\begin{proof}[Proof of the claim]
 The LHS of \eqref{FewTranslates}
 equals to
$\DS \int_{\T^{s+1}} \prod_{i=1}^s 1_{-A} (t_j-x) dt_1\dots dt_s dx=
\Leb(A)^s 
$
where the  equality uses the change of variables $u_j=t_j-x.$
\end{proof}
We now apply the claim to the sets $A_N$ defined by \eqref{eq:an0}, $\delta_N=\frac{1}{N^{1+\gamma/5}}$ and  $s=\frac{50}{\gamma}$. Then, using Proposition \ref{lemma.delta} 
$$\int_{[0,1]^s}Leb\Big(\bigcap_{i=1}^s(A_N+t_i)\Big)dt_1\ldots dt_s\leq  N^{-9}.$$

So by Markov's inequality the set 
$$
B_N:=\Big\{(t_1,\ldots,t_s)\in [0,1]^s\;:\; {Leb}_s\Big(\bigcap_{i=1}^s(A_N+t_i)\Big)\leq N^{-6}\Big\},
$$
satisfies
$\DS
Leb_s\Big(B_N\Big)\geq 1-\frac{1}{N^2}.
$
We now define the set $\mathbb{D}$ as  
$$
\mathbb{D}:=\bigcup_{m=1}^{+\infty} \bigcap_{N\geq m}B_N
$$
Then $\mathbb{D}$ is a full measure subset of $[0,1]^s$ that satisfies the condition of Lemma \ref{lemma.G2}.  \end{proof}

\medskip

We finally define the set $\cC$ that satisfies the requirements of 
Proposition \ref{prop.choice.c}
 \begin{equation}\label{eq:cc}
\cC:=\Big\{(c_1,\ldots,c_{s+3})\in [0,1]^s:(c_{i_1},\ldots,c_{i_s})\in \mathbb{D} \cap \cF \text{  for every  } i_1,\ldots,i_s\in \{1,\ldots, s+3\}\Big\}.
\end{equation}

It follows from Lemmas \ref{lemma.G1} and  \ref{lemma.G2} that ${\rm Leb}(\cC)=1$ and that it satisfies $(\cG1)$ and $(\cG2)$. We still have to prove $(\cG3)$.

\begin{lemma}\label{lem:sep}
For sufficiently large $T>0$, for every ${\bf x}\in \T^2$ there exists $(i_1,\ldots,i_s)\subset \{1,\ldots,s+3\}$ such that ${\bf x}\in  S(T,(i_v)_{v=1}^s)$.
\end{lemma}

\begin{proof} Notice that there is at most one $c_j$ such that  ${\bf x}\in V_\kappa(\bar{c}_j)$ and at most one $c_{j'}$ such that  ${K_T}{\bf x}\in V_\kappa(\bar{c}_{j'})$. It is thus enough to show that there is at most one $c_{j''}$ such that  
$$
\{x+N(x,w,T)\alpha+u\alpha\}_{u\in [-T^\gamma,T^\gamma]} \in  \Big[-\frac{1}{2T^\gamma\ln^{10}T}+ c_{j''}, c_{j''}+\frac{1}{2T^\gamma\ln^{10}T}\Big].
$$
Indeed, if there were two different $c_{j_1}$ and $c_{j_2}$, then for some $|k|\leq 2T^\gamma$, 
$$
\|c_{j_1}-c_{j_2}+k\alpha\|\leq \frac{1}{4T^\gamma\ln^{10}T}.
$$
Let $n\in \N$ be unique such that $T^\gamma \in [q_n,q_{n+1}]$. The above condition would, by the fact that $\alpha\in \cD$ imply that (see \eqref{eq:gcn})
$\DS
G(c_{j_1},n)\cap G(c_{j_2},n)\neq \emptyset.
$
This however contradicts the fact that any $s$- tuple of coordinates of the vector $(c_1,\ldots, c_{s+3})$ belongs to $\cF$ and in particular to $F_n$ if $n$ is sufficiently large (see \eqref{eq:fn}).
\end{proof}

Proposition  \ref{prop.choice.c} follows from Lemmas \ref{lemma.G1}, \ref{lemma.G2} and \ref{lem:sep}. \hfill $\Box$

\subsection{Proof of Proposition \ref{prop.choice.tau}}
\label{SSprop.choice.tau}

From now on we assume that $(K_t)$ is as in \eqref{fsum} with $(c_1,\ldots,c_{s+3})$ {\it good}.

For $T>0$ and ${\bf x}=(x,w)\in \T^2$ recall that $N(x,w,T)$ denotes the number of returns of $\{K_t(x,w)\}_{t\leq T}$ to
 $\T\times \{0\}$. Then $N(x,w,T)\leq C T$, where $C=(\inf_\T f)^{-1}$.
 
The following lemma allows us to relate the ergodic sums of $\tau_j$ to those of $\bar f_0(\cdot-c_j)$. It relies on the fact that $\tau_j$ equals $1$ in a neighborhood of $\bar{c}_j$ and equals $0$ in the neighborhood of $\bar{c}_i$, for $i\neq j$.  The control of the error is due to the fact that $\a \in \cD$. 

\begin{lemma}
\label{LmErgSumRetT}
Let $\tau=\tau_j$ with $j\in \{1,\ldots, s+3\}$. Then for $T$ sufficiently large
$$
\Big|\sum_{u=0}^{N(x,w,T)-1}\int_0^{f(x+u\alpha)}\tau(K_t(x+u\alpha,0))dt-\sum_{u=0}^{N(x,w,T)-1}\bar{f}_0(x+u\alpha-c_j)\Big|\leq \ln^4T. 
$$
where $\bar{f}_0=\bar{f}-\int_\T \bar{f} dLeb$ and $\bar{f}$ is as in \eqref{eq:asy2}.
\end{lemma}

\begin{proof} 
Denote 
$$\varphi_{\tau}(x)=\int_0^{f(x)}\tau(K_t(x,0))d{t}, \quad\text{and}\quad 
f_j=\bar{f}(x-c_j).$$

For a function $g:\T \to \R$ define 
$$g_\kappa=1_{x\in [-\kappa,\kappa]} g, \quad g^\kappa=1_{x\notin [-\kappa,\kappa]} g.$$

Since $\tau=1$ on $V_\kappa(c_j)$, we have the following identity
$$
\varphi_\tau(x)=(\varphi_\tau)_{\kappa}(x)+(\varphi_\tau)^{\kappa}(x)= (f_j)_{\kappa}(x)+ \int_{0}^{f(x)}(\tau-1)(K_tx)1_{-[\kappa,\kappa]}(x)d{t}+(\varphi_\tau)^{\kappa}(x)
$$
Notice that the functions $h_\kappa(x)=\int_{0}^{f(x)}(\tau-1)(K_tx)1_{-[\kappa,\kappa]}(x)d{t}$ and $(\varphi_\tau)^{\kappa}$ belong to the class $BV(\T)$. 
With this notation, denoting also $N=N(x,w,T)$, we need to bound  
\begin{align*} S_N \varphi_\tau-S_N f_j +N\int_{\T} f_j d\lambda&= S_N(h_\kappa)+
S_N \varphi_\tau^\kappa-S_N f_j^\kappa -N\left(\int_{\T} h_\kappa d\lambda+\int_{\T} \varphi_\tau^\kappa d\lambda-\int_{\T} f_j^\kappa d\lambda \right)
 \end{align*} 
where in the LHS we used that  $\int_{\T} \varphi_{\tau} d\lambda=0$ because $\int_{\T} \tau d\lambda=0$. 
Since $h_\kappa$,  $\varphi_{\tau}^{\kappa}$ and $f_j^{\kappa}$ are all of bounded variation, 
the bound of the lemma follows from Lemma \ref{LmSumBV}.
\end{proof}
 
 \noindent We will often use the following decomposition of the orbital integral: for ${\bf x}\!=\!(x,w)\!\in\! \T^2$
 
 \begin{multline}\label{mult:dec}
\int_{0}^T\tau(K_t({\bf x}))dt=\sum_{i=0}^{N(x,w,T)-1}\int_0^{f(x+i\alpha)}\tau(K_t(x+i\alpha,0))dt\\
-\int_0^w\tau(K_t(x,0))dt+ \int_0^{T+w-S_{N(x,w,T)}(f)(x)}\tau(K_t(x+N(x,w,t)\alpha,0))dt.
\end{multline}

 \begin{corollary} \label{LmErgSumRetT2}
Let $\tau=\tau_j$ with $j\in \{1,\ldots, s+3\}$. Then with $N=N(x,w,T)$
\begin{align*}
\left|\int_{0}^T\tau(K_t({\bf x}))dt-\sum_{i=0}^{N-1} \bar{f}_0(x+u\alpha-c_j) \right| &\leq C\cdot\Big(|\bar{f}(x-c_j)|+|\bar{f}(x+N\alpha-c_j)|+\ln^4T\Big)\end{align*}
 \end{corollary}
 
 \begin{proof} {We use \eqref{mult:dec}.  By Lemma \ref{LmErgSumRetT}, we can
  replace the first term of the LHS by $\DS \sum_{i=0}^{N-1} \bar{f}_0(x+u\alpha-c_j)$. For the last two terms 
  in \eqref{mult:dec}
 we} use that $\tau_j$ vanishes in the neighborhoods of $\bar{c}_i$ for $i\neq j$, that 
 $ w \in [0,f(x))$ and $T+w-S_Nf(x)\in [0,f(x+N\alpha))$.
 \end{proof}

We will split the proof of Proposition \ref{prop.choice.tau} into two cases: {\bf Case 1} in which we will consider points $(x,w)$ with very close visits to the set of singularities, and {\bf Case 2} which covers the complimentary points. Fix $\eps\ll 1$. 
\smallskip

  \noindent {\bf Case 1.} $N(x,w,T)<T^{1-\eps}$.

\begin{proposition} \label{prop.Nsmall}  If $(x,w)$ is such that $N(x,w,T)<T^{1-\eps}$ then 
 \begin{equation} \label{eq.TT} \max_{j\in \{1,\ldots,s+3\}} \int_0^T\tau_j(K_tx)dt \geq \eps^2 T.\end{equation}
\end{proposition}

In all of this section we sometimes simply denote $N(x,w,T)$ by $N$. First observe the following.

\begin{lemma} \label{lemma.Nsmall2} For any $\eps>0$, if $T$ is sufficiently large and  $N(x,w,T)<T^{1-\eps}$, there exists $j\in \{1,\ldots, s+3\}$ and $u\in [0,N(x,w,T)]$ such that $\|x+u\alpha-c_j\|\leq \frac{1}{\eps T^{1/\gamma}}$. 
\end{lemma}

\begin{proof} Assume that for any $j\in \{1,\ldots, s+3\}$ and any $u\leq N$,  $\|x+u\alpha-c_j\|\geq \eps^{-1} T^{-1/\gamma}$. This can be expressed as $x_{min,N+1}\geq \eps^{-1} T^{-1/\gamma}$.

Observe that by definition $0\leq T+w-S_{N}f(x)\leq f(x+N\alpha)$. By Lemma \ref{lem:DK0} (applied to all the functions $\bar{f}(\cdot-c_j)$) it follows that if $\eps$ is sufficiently small
$$
T\leq S_{N+1}(f)(x)\leq C(N+ f(x_{min,N+1}))\leq  CT^{1-\eps}+ 2\eps^\gamma T\leq T/2,$$
a contradiction.\end{proof} 

\begin{proof}[Proof of Proposition \ref{prop.Nsmall}] 

We will split the proof in several cases. 
\medskip 

\noindent {\it Case I. There is $j\in \{1,\ldots,s+3\}$ and $u\in [1,N-1]$ such that $\|x+u\alpha-c_j\|\leq \eps^{-1} T^{-\frac{1}{\gamma}}$}. 

Lemma \ref{lem:DK0} applied to the function $f_j:=\bar{f}_0(\cdot-c_j)$ then implies that 
\begin{equation} \label{eqj} S_{N}f_j(x)\geq \eps^\gamma T/2.
\end{equation}

Since $\alpha \in \cD$, we have that for $u'\in \{0,N\}$, $\|x+u'\alpha-c_j\|\geq N^{-1-\eps} \geq (CT)^{-1-\eps}$. Corollary \ref{LmErgSumRetT2} and equation \eqref{eqj} then imply 
 \begin{equation}
 \label{eq.j2} \int_0^T\tau_j(K_tx)dt  \geq \eps^\gamma T/4.\end{equation}

\medskip

\noindent {\it Case II. For any $j\!\in\!\{1,\ldots,s+3\}$ and any $u\in [1,N-1]$ we have $\|x+u\alpha-c_j\|> \eps^{-1} T^{-\frac{1}{\gamma}}$}.

We split this case in several sub-cases. 

\medskip 
\noindent{\it Case II.1. There exists $j\!\in\!\{1,\ldots,s+3\}$ such that  $\|x-c_j\|\leq \eps^{-1} T^{-1/\gamma}$ 
and $f(x)\!-\!w\!>\!\eps T$.}

By the diophantine assumption $\alpha \in \cD$ it follows that for any sufficiently large $M\in \N$ and for any $j$ the set $\{x+u\alpha-c_j\}_{0\leq u\leq M}\cap [-M^{-1-\eps},M^{-1-\eps}]$ is at most a singleton. In particular, since $N<T^{1-\eps}$ it follows that 
\begin{equation}\label{eq1N} \forall u \in [1,N] : \|x+u\alpha-c_j\|\geq     N^{-1-\eps}.
\end{equation}  

We have, since $\tau_j=1$ on all but a bounded part of the fiber above $x$
\begin{equation}\label{eqj22} \int_0^{f(x)-w} \tau_j(K_tx)dt  \geq \eps T/2,\end{equation}
while \eqref{eq1N} and exactly the same argument  as in Corollary \ref{LmErgSumRetT2} imply that 
\begin{equation*} \left| \int_{f(x)-w}^T \tau_j(K_tx)dt\right| \leq \eps T/4.
\end{equation*}
Hence \eqref{eqj22} allows to conclude that 
$\int_0^{T} \tau_j(K_tx)dt  \geq \eps T/4.$

\medskip 
\noindent{\it Case II.2. There exists $j \in   \{1,\ldots,s+3\}$ such that  $\|x-c_j\|\leq \eps^{-1} T^{-1/\gamma}$ and $f(x)-w\leq \eps T$.}

Since for any $j\in \{1,\ldots,s+3\}$ and any $u\in [1,N-1]$ we have $\|x+u\alpha-c_j\|> \eps^{-1} T^{-\frac{1}{\gamma}}$, Lemma \ref{lem:DK0} implies that $S_{N-1}(f)(x+\alpha)\leq \sqrt{\eps}T$.

This, and the fact that $f(x)-w\leq \eps T$, imply that for times $t\leq [T_0, T]$, with $T_0\leq \sqrt{\eps}T+\eps T$ the orbit of $(x,w)$ $K_t(x,w)$ is in the last fiber above $x+N\alpha$. In particular $\|x+N\alpha-c_{j'}\|\leq \eps^{-1} T^{-\frac{1}{\gamma}}$ for some $j' \neq j$. Since $S_{N-1}f(x+\alpha)\leq \sqrt{\eps}T$, the decomposition in equation \eqref{mult:dec} used as in Corollary \ref{LmErgSumRetT2} implies that
  $\int_{0}^T \tau_{j'}(K_tx)dt \geq \eps^{2}T.$

\medskip 
\noindent{\it Case II.3. For all $j \in   \{1,\ldots,s+3\}$,  $\|x-c_j\|> \eps^{-1} T^{-1/\gamma}$.}
In this case we have the following bound: $\|x+u\alpha-c_j\|\geq \eps^{-1} T^{-1/\gamma}$ for $u\in[0,N-1]$ and all $j\leq s+3$. Then  Lemma \ref{lem:DK0} implies that $S_{N}f(x)\leq \sqrt{\eps}T$. Hence, after time $T_0\leq S_N(f)(x)\leq \sqrt{\eps}T$ the orbit of $(x,w)$ up to time $T$ moves vertically up in the last fiber above $x+N\alpha$. This in particular means that  $\|x+N\alpha-c_{j'}\|\leq \eps^{-1} T^{-\frac{1}{\gamma}}$. We then conclude as in Case~II.2.
\end{proof}

\noindent
{\bf Case 2.} $N(x,w,T)>T^{1-\eps}$.


 We will now split the phase space $\T^2$ into finitely many disjoint sets, and we will estimate the ergodic integrals separately on them.  Let $V_\kappa (\bar{c})$ be the 
 $\kappa$ neighborhood of $\bar{c}\in \T^2$ (or $\T$). Let $i_1,\ldots,i_s$ be different elements from $\{1,\ldots,s+3\}$.

 Because $(c_1,\ldots,c_{s+3})$ is {\it good}, it follows from $(\cG2)$ that Proposition \ref{prop.choice.tau} is implied by 
 \begin{proposition} \label{prop.tau.S}  For every choice 
  $(i_1,\ldots,i_s)\subset \{1,\ldots,s+3\}$
\begin{equation}\label{eq:sti'}
\Leb_{\T^2}\Big(\{{\bf x}\in S(T,(i_v)_{v=1}^s)\;:\; \max_{j \in \{i_1,\ldots,i_s\}} \Big\|\int_{0}^T \tau_j(K_t{\bf x})dt\Big\|\leq T^{\gamma^2+\eps/2}\}\Big)\leq T^{-5}.
\end{equation}
\end{proposition}


We will now further partition the set $S(T,\{i_v\}_{v=1}^s)$ into level sets according to the values 
of $N(x,w,T)$. For $u\in [T^{1-\gamma-\eps},CT^{1-\gamma}]$, $u \in \N$ let
\begin{equation}
\label{DefW}
W(u,\{i_v\}_{v=1}^s):=\{{\bf x}=(x,w)\in S(T,\{i_v\}_{v=1}^s): N(x,w,T)\in [uT^\gamma,(u+1)T^\gamma)\}.
\end{equation}

Since $u\leq CT^{1-\gamma}$ to establish Proposition \ref{prop.tau.S} it is sufficient to prove

 \begin{proposition} \label{prop.tau.U}  For $T$ sufficiently large, for every 
 $s$-tuple  
  $(i_1,\ldots,i_s)\subset \{1,\ldots,s+3\}^s$, for every  $u\in [T^{1-\gamma-\eps},CT^{1-\gamma}]$   
 \begin{equation}\label{eq:sti}
\Leb_{\T^2}\Big(\{{\bf x}\in W(u,\{i_v\}_{v=1}^s)\;:\; \max_{j \in \{i_1,\ldots,i_s\}} \Big\|\int_{0}^T \tau_j(K_t{\bf x})dt\Big\|\leq T^{\gamma^2+\eps/2}\}\Big)\leq T^{-6+6\eps}.
\end{equation}
\end{proposition}

\begin{proof}[Proof of Proposition \ref{prop.tau.U}]
Since the indexing set $\{i_v\}_{v=1}^s$ is fixed, we will drop it from the notation of $W(u)$. 

In all this proof, we will suppose that ${\bf x}=(x,w)\in W(u)$ for some fixed choice of $u\in [T^{1-\gamma-\eps},CT^{1-\gamma}]$ and $i_1,\ldots,i_s$. 
In the proof, we may use $\tau$ to denote any of the functions $\{\tau_i\}$ for $i\in \{i_1,\ldots,i_s\}$.  From the decomposition \eqref{mult:dec}  and \eqref{eq:awaysing}, we have for any $j \in \{i_1,\ldots,i_s\}$
\begin{equation} \label{eq.00} \left|\int_{0}^T\tau_j(K_t({\bf x}))dt- \sum_{u=0}^{N(x,w,T)-1}\bar{f}_0(x+u\alpha-c_j) \right|\leq (\ln T)^5.\end{equation}

We have 
\begin{equation}\label{eq:split2}
\sum_{u=0}^{N(x,w,T)-1}\bar{f}_0(x+u\alpha-c_j)=S_{[uT^\gamma]}(\bar{f}_0)(x+c_j)+ 
\sum^{{ N(x,w,T)}}_{{ \ell=[uT^\gamma]}}\bar{f}_0(x+\ell\alpha-c_j)
\end{equation}

In the following lemma we use $(\cG3)$ to see that the contribution of the second summand is of lower order, so that we can in the sequel focus only on the first one. 
\begin{lemma}\label{lem:lowo} For every $x\in \T$ so that $ (x, 0)\in W(u)$ and for every $j\in \{i_1,\ldots, i_s\}$,
$$
\Big|\sum_{\ell=[uT^\gamma]}^{N(x,w,T)}\bar{f}_0(x+\ell\alpha-c_j)\Big|<T^{\gamma^2+\eps^2},
$$
for sufficiently large $T>0$.
\end{lemma}
\begin{proof}
We have
$$ 
\sum_{\ell=[uT^\gamma]}^{N(x,w,T)}\bar{f}_0(x+\ell\alpha-c_j)=S_{N(x,w,T)-[uT^\gamma]}(\bar{f}_0(\cdot-c_j)(x+[uT^\gamma]\alpha).
$$
By \eqref{eq:awaysing2} and since $|[uT^\gamma]-N(x,w,T)|\leq T^\gamma$ it follows that 
$$
\left\{x+[uT^\gamma]\alpha+u\alpha\right\}_{u=0}^{N(x,w,T)-[uT^\gamma]}\bigcap 
\left[-\frac{1}{2T^\gamma\ln^{7}T}+ c_j, c_j+\frac{1}{2T^\gamma\ln^{7}T}\right]=\emptyset. 
$$
This and Lemma \ref{lem:DK0} gives
$\DS
|S_{N(x,w,T)-[uT^\gamma]}(\bar{f}_0(\cdot-c_j)(x+[uT^\gamma])\alpha)|={\rm O} (T^{\gamma^2}\ln^{200}T).
$
\end{proof}

Since $uT^\gamma \geq T^{1-\eps}$ implies that $[uT^\gamma]^{\gamma^2+\eps}\geq T^{\gamma^2+\eps/2}$, it follows from   \eqref{eq:split2} and Lemma \ref{lem:lowo}, that  \eqref{eq:sti} holds if we show that  
\begin{equation}\label{eq:sti2}
Leb_1\Big(\{x\in \T\;:\; \max_{j \in \{i_1,\ldots,i_s\}} |S_{[uT^\gamma]}(\bar{f}_0)(x-c_j)|\leq [uT^\gamma]^{\gamma^2+\eps}\}\Big)\leq T^{-6+6\eps}.
\end{equation}
By $(\cG2)$ in the definition of  $(c_{1},\ldots, c_{s+3})$ being good, we have that the LHS of \eqref{eq:sti2} is less than $ (uT^\gamma)^{-6}\leq T^{-6+6\eps}$. This gives \eqref{eq:sti}. The proof of Proposition \ref{prop.tau.U} is now finished. 
\end{proof}

Proposition \ref{prop.choice.tau} 
follows from Propositions \ref{prop.Nsmall} and 
\ref{prop.tau.S}. \hfill $\Box$

With Propositions  \ref{prop.choice.c}  and 
\ref{prop.choice.tau}  proved, Theorem \ref{prop:sch} is established. We will now see how the construction can be modified in order to yield the real analytic examples of Theorem \ref{thm:main2}. 

\section{ Slowly parabolic Kochergin flows. The analytic case} \label{sec.analytic}

As the Kochergin flow $(K_T)$ and the fiber $\R^d$ action $(G_{\bf t})$ can both be taken to be analytic, the only thing we need to show in Proposition \ref{prop:1} is that there exists an analytic ${\bf \bar{\tau}}$ satisfying $S2$ and that there exists an analytic function $H$ with $\sigma^2(H)\neq 0$. Let $c_1,\ldots, c_{s+3}\in \mathcal{C}$ be the fixed points of $(K_T)$. Choose a large integer $L$.
For $i\leq s+3$ let $\bar{\tau}_i:\T^2\to\R$ be an analytic mean zero function such that $\bar{\tau}_i(c_i)=1$, $\bar{\tau}_i(c_j)=0$ and such that all derivatives of $\bar{\tau}_i$ up to order $L$ vanish at all the $\{c_i\}_{i=1}^{s+3}$ and define $\bar{\btau}(x)=(\bar{\tau}_1(x),\ldots,\bar{\tau}_{s+3}(x))$. Recall that  ${\btau}\in C^\infty(\T^2)$ was defined in \S \ref{sec.tau.choice}.

We have the following approximation result.

\begin{lemma}\label{lem:alcob} If $L$ is sufficiently large, then
$$
\nu\Big(\{{\bf x}\in \T^2\;:\; \Big\|\int_0^T[\btau-\bar{\btau}](K_t{\bf x})dt\Big\|>\ln^5 T  \}\Big)=o(T^{-100}).
$$
\end{lemma}
The  lemma together with the fact that $\btau$ satisfies Proposition \ref{prop.choice.tau}, immediately gives that $\bar{\btau}$ satisfies $S2$. So we only need to prove the lemma.
\begin{proof}[Proof of Lemma \ref{lem:alcob}]
Notice that for each $i\leq s+3$ the function $\tau_i-\bar{\tau}_i$ is a mean zero $C^\infty$ function which is zero at all the $c_i$ and also the derivatives up to order $L$ are zero  at all the $c_i$. Let $f$ be the roof function and let $\psi_i:\T\to \R$,  $\psi_i(x)=\int_0^{f(x)}[\tau_i-\bar{\tau}_i](x)d\lambda$. It follows from Theorem 4.1. in \cite{FM} that if $L$ is sufficiently large, then $\psi_i$ has {\em logarithmic} singularities at $(c_i)$. Define $B_{good}(T)$ to be the set of $(x,w)\in X$ such that $\{K_t(x,w):t\leq T\}$ does not intersect the set $\bigcup_{j=0}^{s+3}[-T^{-200/\gamma}+c_j,c_j+T^{-200/\gamma}]$. Clearly $Leb_2(B_{good}(T))=o(T^{-100})$. 
 Now, as shown in the proof of Lemma 7.1. in \cite{FU}
 \footnote{The statement of the lemma only gives an upper bound 
 $$\left|\int_0^T[\tau_i-\bar{\tau}_i](K_t(x,w))dt\right|\leq \left| S_{N(x,w,T)}(\psi_i)(x)\right|+\ln^2 T,$$ but in the proof the authors in fact have the identity $\int_0^T[\tau_i-\bar{\tau}_i](K_t(x,w))dt=S_{N(x,w,T)}(\psi_i)(x)+$Error and the upper bound on the error for $(x,w)\in B_{good}(T)$  is $|\mathrm{Error}|\leq \ln^2T$.}, we have for sufficiently large $T$ and $(x,w)\in B_{good}(T)$ 
 $$\left| \int_0^T[\tau_i-\bar{\tau}_i](K_t(x,w))dt- S_{N(x,w,T)}(\psi_i)(x)\right|\leq \ln^2 T.
$$

Now, since $\psi_i$ has logarithmic singularities,  $N(x,w,T)\leq CT$ 
{ and 
 $\alpha\in \cD,$} 
 it follows that for every $(x,w)\in B_{good}(T)$ and $T$ sufficiently large 
$$
|S_{N(x,w,T)}(\psi_i)(x)|\leq \ln^4T,
$$
(see eg. Proposition 5.2. in \cite{CW}).
This finishes the proof.
\end{proof}

\appendix
\section{Variance}
We assume the flow $(F_T,M,\zeta)$ is given by  \eqref{eq:Tcont}, where $(K_T)$ is
  the Kochergin flow, 
  the fiber $\R^d$ action $(G_{\bf t})$ is exponentially mixing of all orders and the analytic cocylce ${\bf \bar{\tau}}$ defined in Section \ref{sec.analytic}. Here, we prove the following. 
\begin{lemma}
\label{LmVar}
There exists an analytic observable $H:M\to \reals$ such that $\sigma^2(H)\neq 0$.
\end{lemma}

We note that with a little effort we could extend the proof of Lemma 8.2 in \cite{DDKN}
and show that $\sigma^2$ is not identically equal to zero on $C^1.$  The result then follows from
the density of analytic functions in $C^1.$ However, to make the paper more self contained we 
sketch the argument below.

In the proof below
we shall use the following formula for $\sigma^2(H)$ obtained
 in \cite{DDKN}
\begin{equation}
\label{VarFormula}
\sigma^2(H)=
 \int_{-\infty}^\infty \int_{X}\int_{Y} \brH(x, y) \brH(K_t x, G_{\tau_t(x)} y) d\nu(y) d\mu(x)  dt 
\end{equation}
where $\brH(x,y)=H(x,y)-\int H(x,y) d\nu(y).$ 

This series converges because for each $y\in Y$ we have $\int \brH(x,y) d\nu(y)=0$ and so for each 
$x\in X$ we have due to the exponential mixing of $G$ that 
$$ \left| \int\brH(x, y) \brH(K_t x, G_{\tau_t(x)} y) d\nu(y) \right|\leq 
C\|\brH\|^2 \left[e^{-c\ln^2 t}+1_{|\tau_t(x)|\leq \ln^2 t}\right]. $$
Integrating with respect to $x$ and using $S2$ we get
$$ \left| \iint\brH(x, y) \brH(K_t x, G_{\tau_t(x)} y) d\nu(y) d\mu(x)  \right|\leq 
\frac{C\|\brH\|^2}{t^5} $$
which implies the convergence of \eqref{VarFormula}.

\begin{proof}
[Proof of Lemma \ref{LmVar}]
Given a point $a\in [0,1]^2$ and $\delta>0$
 consider the following function on $\reals^2$
 $$ \Theta(x,a)=\sum_{m\in \integers^2} e^{-\|x-a-m\|^2/\delta^2}.$$
Since $\Theta$ is $\integers^2$ periodic,  it descends to a function on $\T^2$
 which will also be denoted by~$\Theta.$ 
 
 The following estimates of this function are direct consequences of standard properties of Gaussian integrals:
 
 (P1) $\Theta(x,a)\leq e^{-\delta^{-0.1}}$ 
 if $d(x,a)>\delta^{0.9}.$
 
 (P2) $\|\Theta\|_{L^2}=c \delta^2+O(e^{-\delta^{-0.1}}).$ 
 
 (P3) If $d(a_1, a_2)<0.1$ then
 $\DS \int \Theta(x, a_1) \Theta(x, a_2) dx=c\delta^2 e^{-d(a_1, a_2)^2/2 \delta^2}
 +O(e^{-\delta^{-0.1}}).$
 
 Now fix $x_0\in X$  satisfying $S3$ and consider a function 
 $H(x,y)=\Theta(x, x_0) D(y)$  where $D$ is an analytic zero mean function.
  Note  that $\int H(x,y) d\nu(y)\equiv 0$ so $\brH=H$ in this case. Hence
 $$ \sigma^2(H)=
 \int_{-\infty}^\infty \int_{X}\int_{Y} \Theta(x, x_0) \Theta(K_t x, x_0) D(y) D(G_{\tau_t(x)} y) d\nu(y) d\mu(x)  dt .
 $$
 We now divide the integral with respect to $t$ into three parts
 
 (I) $|t|\leq C\delta^{0.9}.$ 
 Using the Taylor expansion of $\|x(t)-x_0\|$ and $D(y(t))$ as well as 
 (P3) we conclude that the contribution of first time segment equals to $\bc \delta^3(1+o(1)) \|D\|_{L^2}^2 $
 for the appropriate constant $\bc>0.$
 
 (II) $C\delta^{0.9}<|t|<\delta^{-0.9/m}.$ 
 In this case either $\Theta(x, x_0)<e^{-\delta^{-0.1}}$ or $\Theta(K_t x, x_0)<e^{-\delta^{-0.1}}$. 
 Otherwise both $x$ and $K_t x$ would be $\delta^{0.9}$-close to $x_0$ due to (P1), 
 which contradicts $S3$.
 Thus the contribution of the second time segment is $O(e^{-\delta^{-0.1}}).$

(III) $|t|>\delta^{-0.9/m}.$ Fix $t$ and split the integral over $X\times Y$ into two regions.

(a) $\|\tau_t(x)\|<C\ln^2 t.$ 
This set has measure $O(t^{-5})$ due to $S2$. 
 
(b) $\|\tau_t(x)\|\geq C\ln^2 t.$ In this case integrating with respect to $y$ and using exponential mixing
 we conclude that the region (b) contributes $O(e^{-C\ln^2 t})$ which is much smaller than our estimate
 for the region (a). 
 
 It follows that the third time segment contributes
 $$ O\left(\int_{\delta^{-0.9/m}}^\infty \frac{dt}{t^5} \right)=
 O\left(\delta^{3.6/m}\right)=o(\delta^3) $$
 since $m<1.1.$
 
 \noindent Combing the estimates (I)--(III) above we conclude that 
 $\sigma^2(H)=\bc \delta^3 (1+o_{\delta>0} (1))\|D\|_{L^2}^2$,
so it is positive for sufficiently small $\delta$ completing the proof of the lemma. 
 \end{proof}

\section{Ostrovski estimate.}
\label{AppO}
 Several proofs in our paper rely on the following standard estimate. 

\begin{lemma}
\label{LmSumBV}
If $\alpha\in \cD$ where $\cD$ is given by \eqref{eq:cp}  and $h\in BV(\cT)$ then for all $x\in \T$
$|S_N(h)(x)|\leq O(\ln^4 N).$
\end{lemma}

\begin{proof}
Recall  
the classical Denjoy-Koksma inequality: 
if $h\in BV(\T)$ then
for every $x\in \T$ and every $n\in \N$
\begin{equation}\label{eq:dkbv}
|S_{q_n}(h)(x)-q_n\int_\T h(x) d x|\leq 2{\rm Var}(h).
\end{equation}
 To bound $S_N(h)(x)$ for general $N\in\N$ we use the Ostrovski expansion to write \\
 $\DS N=\sum_{k=1}^m b_kq_k$ where $q_m\leq N<q_{m+1}$
 and $b_k\leq a_k$. We then use cocycle identity, \eqref{eq:dkbv} and the fact that  $m=O(\ln N)$ while
our assupmtion \eqref{eq:cp} implies that for $k\leq m$ we have $a_k=O(\log^2 N)$.
\end{proof}

\end{document}